\documentclass[11pt,oneside,english,makeidx]{amsart}
\usepackage[T1]{fontenc}
\usepackage[latin9]{inputenc}
\usepackage[a4paper]{geometry}
\usepackage{color}
\usepackage{babel}
\usepackage{verbatim}
\usepackage{amsthm}
\usepackage{amssymb}
\usepackage[all]{xy}
\usepackage[unicode=true,
 bookmarks=true,bookmarksnumbered=false,bookmarksopen=false,
 breaklinks=false,pdfborder={0 0 1},backref=false,colorlinks=true]
 {hyperref}
\hypersetup{pdftitle={The indeterminacy locus of the Voisin map},
 pdfauthor={ Muratore Giosu\`{e}Emanuele},
 pdfsubject={Algebraic Geometry}}

\makeatletter
\numberwithin{equation}{section}
\numberwithin{figure}{section}
\theoremstyle{plain}
\newtheorem{thm}{\protect\theoremname}[section]
\theoremstyle{plain}
\newtheorem{conjecture}[thm]{\protect\conjecturename}
\theoremstyle{plain}
\newtheorem{lem}[thm]{\protect\lemmaname}
\theoremstyle{plain}
\newtheorem{prop}[thm]{\protect\propositionname}
\theoremstyle{definition}
\newtheorem{defn}[thm]{\protect\definitionname}
\theoremstyle{remark}
\newtheorem{rem}[thm]{\protect\remarkname}
\theoremstyle{remark}
\newtheorem*{claim*}{\protect\claimname}

\newtheorem*{Notation*}{Notation}\subjclass[2010]{Primary 14E05, 53D05; Secondary 32J27}

\usepackage{enumitem}
\newenvironment{myitemize}{\begin{itemize}[leftmargin=*]
\addtolength{\leftmargin}{0in}
}{\end{itemize}}
\usepackage{babel}
\providecommand{\claimname}{Claim}
  \providecommand{\conjecturename}{Conjecture}
  \providecommand{\definitionname}{Definition}
  \providecommand{\lemmaname}{Lemma}
  \providecommand{\propositionname}{Proposition}
  \providecommand{\remarkname}{Remark}
\providecommand{\theoremname}{Theorem}

\makeatother

\providecommand{\claimname}{Claim}
\providecommand{\conjecturename}{Conjecture}
\providecommand{\definitionname}{Definition}
\providecommand{\lemmaname}{Lemma}
\providecommand{\propositionname}{Proposition}
\providecommand{\remarkname}{Remark}
\providecommand{\theoremname}{Theorem}

\begin{document}
\global\long\def\C{\mathbb{\mathbb{C}}}%
 
\global\long\def\R{\mathbb{\mathbb{R}}}%
 
\global\long\def\Q{\mathbb{\mathbb{Q}}}%
 
\global\long\def\Z{\mathbb{\mathbb{Z}}}%
 
\global\long\def\N{\mathbb{\mathbb{N}}}%
 
\global\long\def\k{\Bbbk}%
 
\global\long\def\bP{\mathbb{P}}%

\global\long\def\e{\epsilon}%
 
\global\long\def\O{\mathcal{O}}%
 
\global\long\def\X{\mathcal{X}}%
 
\global\long\def\dual{\!\vee}%
 
\global\long\def\I{\mathcal{I}}%
 
\global\long\def\U{\mathcal{U}}%

\global\long\def\Ind#1{\mathrm{Ind}(#1)}%

\global\long\def\Effc#1#2{\overline{\mathrm{Eff}}_{#1}(#2)}%
 
\global\long\def\EffcX#1{\overline{\mathrm{Eff}}(#1)}%
 
\global\long\def\Eff#1#2{\mathrm{Eff}_{#1}(#2)}%
 
\global\long\def\Big#1#2{\mathrm{Big}_{#1}(#2)}%
 
\global\long\def\Neff#1#2{\mathrm{Nef}_{#1}(#2)}%
 
\global\long\def\Nef#1{\mathrm{Nef}(#1)}%

\global\long\def\Hom{\mathit{\mathrm{Hom}}}%
 
\global\long\def\Ker{\mathrm{Ker}}%
 
\global\long\def\rk{\mathrm{rk}}%
 
\global\long\def\Spec#1{\mathrm{Spec}(#1)}%
 
\global\long\def\Hilb#1{\mathrm{Hilb}^{#1}}%

\global\long\def\Num{\mathrm{Num}}%
 
\global\long\def\Rat{\mathrm{Rat}}%
 
\global\long\def\Alg{\mathrm{Alg}}%
 
\global\long\def\P#1{\mathbb{\mathbb{P}}^{#1}}%
 
\global\long\def\Pn{\mathbb{\mathbb{P}}^{n}}%
 
\global\long\def\Pw{\mathbb{\mathbb{P}}(\mathbf{w})}%
 
\global\long\def\G{\mathbb{\mathbb{G}}}%

\global\long\def\g{\mathfrak{g}}%
 
\global\long\def\h{\mathfrak{h}}%
 
\global\long\def\b{\mathfrak{b}}%
 
\global\long\def\p{\mathfrak{p}}%

\global\long\def\Yb{Y_{b}}%
 
\global\long\def\Xb{X_{b}}%
 
\global\long\def\ra{\rightarrow}%
 
\global\long\def\rra{\dashrightarrow}%
 
\global\long\def\less{\backslash}%
 
\global\long\def\FF{F\times F}%
 
\global\long\def\dom#1{\mathrm{dom}(#1)}%
 
\global\long\def\sd{|}%
\global\long\def\wt#1{\widetilde{#1}}%
 
\global\long\def\td#1{\tilde{#1}}%
 
\global\long\def\Io{I^{\circ}}%

\title{The indeterminacy locus of the Voisin map}
\author{Giosuè Emanuele Muratore}
\date{\today}
\begin{abstract}
Beauville and Donagi proved that the variety of lines $F(Y)$ of a
smooth cubic fourfold $Y$ is a hyperkähler variety. Recently, C.
Lehn, M. Lehn, Sorger and van Straten proved that one can naturally
associate a hyperkähler variety $Z(Y)$ to the variety of twisted
cubics on $Y$. Then, Voisin defined a degree 6 rational map $\psi:F(Y)\times F(Y)\rra Z(Y)$.
We will show that the indeterminacy locus of $\psi$ is the locus
of intersecting lines.
\end{abstract}

\address{Università Degli Studi Roma Tre\\
Dipartimento di Matematica e Fisica\\
Largo San Murialdo 1, 00146 Roma Italy.}
\address{Université de Strasbourg\\
IRMA\\
7 Rue René Descartes, 67000 Strasbourg France.}
\email{gmuratore@mat.uniroma3.it}
\keywords{hyperkähler, irreducible symplectic variety, rational map.}
\maketitle

\section{Introduction}

It is a classical result that a manifold with a Ricci flat metric
has trivial first Chern class, and by Bogomolov's decomposition \cite{MR0345969,MR730926}
such manifolds have a finite étale cover given by the product of a
Torus, Calabi-Yau varieties and hyperkähler varieties. Hyperkähler
manifolds are interesting in their own and they are the subject of
an intensive research. The first examples are $K3$ surfaces, and
Beauville proved in \cite[Th\' eor\` emes 3 and 4]{MR730926} that
for any $n\ge0$ the Hilbert schemes of points $X^{[n]}$ where $X$
is a $K3$ surface or the generalized Kummer varieties $K^{n}A$ associated
to an Abelian surface $A$, are hyperkähler varieties. Any hyperkähler
variety that is a deformation of $X^{[n]}$ where $X$ is a $K3$
surface (respectively, a generalized Kummer variety) is called $K3^{[n]}$-type
(respectively, $K^{n}A$-type). Those examples are particularly interesting
because they permit to construct hyperkähler varieties of any even
complex dimension. Later O'Grady in \cite{MR1966024,MR1703077} constructed
two new examples in dimension 6 and 10 of hyperkähler varieties that
are not deformation of known types. There are few explicit complete
families of hyperkähler manifolds of $K3^{[n]}$-type. Beauville and
Donagi proved in \cite[Proposition 1]{MR818549} that the variety
of lines $F(Y)$ of a smooth cubic fourfold $Y\subseteq\P 5$ is an
hyperkähler variety of $K3^{[2]}$-type. Another example was given
much more recently by C. Lehn, M. Lehn, Sorger, and van Straten in
\cite{1305.0178}. They observed that if $F_{3}(Y)$ is the moduli
space of generalized twisted cubic curves on $Y$, then $F_{3}(Y)$
is a $\P 2$-fibration over a smooth variety $Z'(Y)$. Moreover, there
is a divisor in $Z'(Y)$ that can be contracted and this contraction
produces a hyperkähler variety $Z(Y)$. This variety is of $K3^{[4]}$-type
by \cite[Corollary]{MR3709062} or \cite[Corollary 6.2]{1504.06406}.

On the other hand the study of $k$-cycles on smooth complex projective
varieties is a classical subject and it is very interesting on hyperkähler
manifolds with respect to several regards. For example, while it is
a classical result that the cone of nef divisors is contained in the
cone of pseudoeffective divisors, in general $\Neff kX\nsubseteq\Effc kX$
for $2\le k\le\dim X-2$. The first example of such phenomenon was
given in \cite{MR2862063}. Later, Ottem proved that if the cubic
$Y$ is very general, then the second Chern class $c_{2}(F(Y))$ of
the Fano variety of lines in $Y$ is nef but it is not effective \cite[Theorem 1]{1505.01477}.

It is known, due to Mumford's Theorem \cite[Theorem]{MR0249428},
that for a projective hyperkähler variety $X$ of dimension $2n$
the kernel of the cycle map $cl:A^{2n}(X)\ra H^{4n}(X)$ is infinite-dimensional
(see \cite[III.10]{MR1997577} for more details). Nevertheless Beauville
made the following
\begin{conjecture}[\cite{MR2187148}]
Let $X$ be a projective hyperkähler manifold. Then the cycle class
map is injective on the subalgebra of $A^{*}(X)$ generated by divisors. 
\end{conjecture}

See \cite{MR3524175} for an introduction to these topics. On the
other hand Shen and Vial in \cite{MR3460114} used a codimension 2
algebraic cycle to give evidence for the existence of a certain decomposition
for the Chow ring of $F(Y)$ for a very general cubic fourfold $Y$.

Voisin constructed in \cite[Proposition 4.8]{MR3524175} a degree
6 rational map 
\[
\psi:F(Y)\times F(Y)\rra Z(Y).
\]
Roughly speaking, the map $\psi$ sends pairs of non-incident lines
$(l,l')\in F(Y)\times F(Y)$ to the (class of the) degree 3 rational
normal curve in the linear system $|L-L'-K_{S_{l,l'}}|$ of the cubic
surface $S_{l,l'}:=\left\langle L,L'\right\rangle \cap Y$. This article
is devoted to the study of the indeterminacy locus of this map. In
particular, we prove the following. 
\begin{thm}
\label{thm:Main}The indeterminacy locus of the Voisin map $\psi:F(Y)\times F(Y)\rra Z(Y)$
is the variety $I$ of intersecting lines in $Y$. 
\end{thm}

We hope that the explicit description of $\Ind{\psi}$ will contribute
to the study of $c_{2}(Z(Y))$, the study of algebraic cycles on $Z(Y)$
and to other aspects of the geometry of $Z(Y)$. We hope to return
to these topics in a future work.

I thank my advisor Gianluca Pacienza for suggesting me this problem
and for introducing me to the study of hyperkähler varieties. Also,
I thank Angelo Lopez for his constant help and for the time he spent
to me during these three years as Ph.D. student. Thanks also to Carlo
Gasbarri for all the suggestions and helpful advises given during
the year at the University of Strasbourg. Furthermore, I appreciate
the help of Charles Vial and Mingmin Shen for explaining to me some
parts of their paper \cite{MR3460114}. I would like to thank also
Edoardo Sernesi and Robert Laterveer for their mathematical advice.
Finally, I thank Claire Voisin for her suggestions, and Christian
Lehn for his hospitality at TU in Chemnitz and for answering all my
questions.

\section{General facts}

A variety $X$ is an integral and separated algebraic scheme over
$\C$. A compact Kähler manifold is hyperkähler if it is simply connected
and the space of its global holomorphic two-forms is spanned by a
symplectic form.

Throughout this paper we will use the following.

\begin{Notation*}$\hspace{1em}$

\begin{myitemize}

\item Given a rational map of varieties $f:A\rra B$, we denote by
$\Ind f$ the complement of the largest open subset of $A$ on which
$f$ is represented by a regular function.

\item $G(k,n)$ is the Grassmannian variety parametrizing $k$-dimensional
linear subspaces of an $n$-dimensional vector space.

\item $T(X)_{x}$ is the tangent space of a smooth variety $X$ at
the point $x$.

\item $Y$ is a smooth cubic fourfold that does not contain a plane.

\item $F$ is the variety of lines on $Y$. By \cite[Proposition 1]{MR818549}
$F$ is a 4-dimensional hyperkähler subvariety of $G(2,6)$.

\item Given a point $l\in F$, the line in $Y$ that it represents
will be indicated with the same letter in uppercase, i.e. $l=[L]$.

\item We denote by 
\begin{equation}
\xymatrix{P:=\left\{ (l,x)\in F\times Y|x\in L\right\} \ar[r]^{\,\,\qquad\qquad\,\,\,\,q}\ar[d]_{p} & Y\\
F
}
\label{eq:F_P_Y}
\end{equation}
the universal family of lines in $Y$.

\item $I$ is the closed subscheme, with reduced structure, of intersecting
lines, i.e. 
\[
I:=\left\{ (l,l')\in\FF\sd L\cap L'\neq\emptyset\right\} .
\]

\item If $X_{1},X_{2}$ are subvarieties in the same projective space,
then $\left\langle X_{1},X_{2}\right\rangle $ is their linear span.

\item If $(l,l')\notin I$, the cubic surface in $Y$ defined by
$L$ and $L'$ is $S_{l,{l'}}:=\left\langle L,L'\right\rangle \cap Y$.

\end{myitemize}\end{Notation*} 
\begin{lem}
\label{lem:CommDiagRatMaps}Let $A,B$ and $C$ be varieties sitting
inside the following commutative diagram 
\[
\xymatrix{A\ar@{-->}[d]_{h}\ar@{-->}[r]^{f} & B\\
C\ar[ru]_{g}
}
\]
where $f$ and $h$ are rational maps and $g$ is a morphism. Then
$\Ind f\subseteq\Ind h$.
\end{lem}

\begin{proof}
Since the composition $g\circ h$ is defined in the domain of $h$,
then $f$ is defined in the domain of $h$ by commutativity.
\end{proof}
\begin{lem}
\label{rmk:connected} Let $pr_{1}:I\ra F$ be the projection to the
first component. Then each fibre of $pr_{1}$ is a surface. Furthermore,
$pr_{1}^{-1}(l)$ is irreducible for general $l\in F$. 
\end{lem}

\begin{proof}
Let $x\in Y$ be a point. Let $C_{x}$ be the subvariety of $F$ parametrizing
lines through $x$. Take a system of coordinates of $\P 5$ such that
$x=[0:0:0:0:0:1]$. Then the equation of $Y$ is $x_{5}^{2}q_{1}+x_{5}q_{2}+q_{3}=0$,
where the polynomial $q_{i}\in\C[x_{0},...,x_{4}]$ is homogeneous
of degree $i$. Since $Y$ is smooth, $q_{1}$ is not the zero polynomial.
The variety $C_{x}$ can now be seen in $\P 4\cong\{x_{5}=0\}$ as
given by $q_{1}=q_{2}=q_{3}=0$. Then $C_{x}$ can be seen as $q_{2}=q_{3}=0$
in $\P 3\cong\{x_{5}=q_{1}=0\}$. Thus, it is connected by \cite[Exercise II.8.4c]{MR0463157}.
Hence, a fibre of $pr_{1}$ is 
\[
pr_{1}^{-1}(l)\cong\bigcup_{x\in L}C_{x}.
\]
It is known that $C_{x}$ is a curve for $x\in Y$, except for finitely
many $x\in Y$ such that $C_{x}$ is a surface \cite[Proposition 2.4]{MR2564370}.
Hence $pr_{1}^{-1}(l)$ is a surface for all $l\in F$. Moreover,
$pr_{1}^{-1}(l)$ is connected for all $l\in F$. Indeed, $pr_{1}^{-1}(l)$
is the union of connected subvarieties $C_{x}$ all of them meeting
at the point $(l,l)$. Furthermore, $pr_{1}^{-1}(l)$ is smooth for
general $l\in F$ by \cite[Section 3, Lemme 1]{MR860684}, hence irreducible. 
\end{proof}
I thank Mingmin Shen for suggesting the following proof to me.
\begin{lem}
\label{lem:IirrCod2}The scheme $I\subseteq\FF$ is irreducible of
dimension 6. 
\end{lem}

\begin{proof}
Let $J:=(q\times q)^{-1}(\Delta_{Y})$. Then $J$ is locally defined
by four equations in $P\times P$, hence each component of $J$ has
dimension at least 6. The map 
\[
p\times p:J\ra I
\]
is surjective and only contracts $\Delta_{P}$ to $\Delta_{F}$. Hence
$p\times p$ is birational and each component of $I$ has dimension
at least 6. Since $pr_{1}^{-1}(l)$ is a surface for all $l\in F$
by Lemma \ref{rmk:connected}, each component of $I$ has dimension
6. Moreover, $pr_{1}^{-1}(l)$ is irreducible for general $l\in F$.
It follows that only one component of $I$ maps surjectively to $F$.
Indeed, let $I_{1}$ be any irreducible component of $I$ such that
$pr_{1}(I_{1})=F$, then we have a surjective map $pr_{1|I_{1}}:I_{1}\ra F$
between two irreducibles varieties of dimension, respectively, 6 and
4. The fibres of this map are $pr_{1|I_{1}}^{-1}(l)=pr_{1}^{-1}(l)\cap I_{1}$.
Thus, for dimensional reasons, the general fibre of $pr_{1|I_{1}}$
is a surface, hence $pr_{1}^{-1}(l)\cap I_{1}$ is a component of
$pr_{1}^{-1}(l)$. As $pr_{1}^{-1}(l)$ is irreducible for general
$l\in F$, it follows that $pr_{1}^{-1}(l)$ has only one component,
therefore
\[
pr_{1}^{-1}(l)=pr_{1}^{-1}(l)\cap I_{1}.
\]
We have proved that $pr_{1}^{-1}(l)\subseteq I_{1}$ for general $l\in F$.
If $I_{2}$ is another irreducible component of $I$ such that $pr_{1}(I_{2})=F$,
the same argument implies that $pr_{1}^{-1}(l)\subseteq I_{2}$ for
general $l\in F$. It follows that $pr_{1}^{-1}(l)\subseteq I_{1}\cap I_{2}$
for general $l\in F$, so $I_{1}=I_{2}$ as they are both irreducible.

Any other component of $I$ different from $I_{1}$ maps to a proper
closed subset of $F$. For dimensional reasons this component must
have dimension at most 5, so that $I=I_{1}$. It follows that $I$
is irreducible and $\dim I=6.$ 
\end{proof}
We define 
\[
\rho:G(2,6)\times G(2,6)\rra G(4,6)
\]
to be the rational map given by the linear span of two general lines
inside a projective space of dimension 5. In the following proposition,
we use an argument already used in \cite[Proposition 20.7]{MR3460114}.
\begin{prop}
\label{prop:IndspanI}The indeterminacy locus of the restricted rational
map\textup{ 
\[
\rho_{|\FF}:\FF\rra G(4,6)
\]
}is the variety $I$. 
\end{prop}

\begin{proof}
This map is clearly defined in the open set of pairs $(l,l')$ such
that $L\cap L'=\emptyset$. In particular 
\[
\Ind{\rho_{|\FF}}\subseteq I.
\]
To prove the other inclusion, let $(l,l')\in I$ be a general point
of $I$, and let $x=L\cap L'$. By surjectivity of the first projection
$I\ra F$, the point $l$ is general in $F$. This implies that $l$
is of first type. That is, the normal bundle to $L$ in $Y$ is 
\[
N_{L/Y}=\O_{\P 1}\oplus\O_{\P 1}\oplus\O_{\P 1}(1),
\]
see, e.g., \cite[Appendix A.3]{MR3460114}. The bundle $N_{L/Y}$
is of rank 3, which implies that we have three linearly independent
vectors that generate $N_{L/Y,x}$. Take $v_{l'}$ to be the image
of a generator of $T(L')_{x}$ in $N_{L/Y,x}$ and choose a basis
$\{v_{l'},w_{0},w_{1}\}$ of $N_{L/Y,x}$.

For every point $(a,b)\in\mathbb{C}^{2}\setminus\{(0,0)\}$, we have
a vector $aw_{0}+bw_{1}\in N_{L/Y,x}$. Since $N_{L/Y}$ is globally
generated, the evaluation at $x$
\[
ev_{x}:H^{0}(L,N_{L/Y})\ra N_{L/Y,x}
\]
is a surjective linear map of vector spaces. By choosing a section
of $ev_{x}$, for each vector $aw_{0}+bw_{1}$ we get a section $s_{a,b}\in H^{0}(L,N_{L/Y})$
such that $s_{a,b}(x)=aw_{0}+bw_{1}$. By \cite[Theorem 6.13]{MR3617981},
we have an isomorphism $H^{0}(L,N_{L/Y})\cong T(F)_{l}$. In particular,
there is a vector in $T(F)_{l}$ induced by $s_{a,b}$. 

Let $C$ be a smooth curve in $F$ such that $l$ is in $C$ and the
tangent direction of $C$ at $l$ is that induced by $s_{a,b}$. We
may assume that for general $l_{c}\in C$, the lines $L_{c}$ and
$L'$ are disjoint in $Y$. If we restrict $\rho_{|\FF}$ to $C\times\{l'\}$,
then $\rho_{|C\times\{l'\}}$ is defined at a general point (hence,
at any point) of $C\times\{l'\}$. So we see that $\rho_{|C\times\{l'\}}$
is well defined at $(l,l')$. For every $l_{c}\in C$ we denote by
$\bP_{l_{c}}$ the linear subspace of $\P 5$ represented by $\rho_{|C\times\{l'\}}(l_{c},l')$.

Let $W$ be the restriction of the universal family $P\ra F$ to $C$.
The flat morphism $W\ra C$ is a $\P 1$-bundle over a smooth curve.
In particular, we can assume that there exists a section $\alpha:C\ra W$
such that $\alpha(l)=x$. Consider the line from $x$ to $\alpha(l_{c})$,
which is clearly contained in $\bP_{l_{c}}$. When $l_{c}$ approaches
$l$, by continuity we get a line through $x$ in $\bP_{l}$. Such
a line is generated, in a natural way, by a lift of $s_{a,b}(x)$
to $T(\P 5)_{x}$. We can get such a lift by considering the following
diagram of vector spaces with exact row and column
\begin{equation}
\xymatrix{ &  &  & 0\ar[d]\\
 &  &  & N_{L/Y,x}\ar[d]\\
0\ar[r] & T(L)_{x}\ar[r] & T(\P 5)_{x}\ar[r] & N_{L/\P 5,x}\ar[r] & 0.
}
\label{eq:Diaglift}
\end{equation}
We deduce that
\[
\rho_{|C\times\{l'\}}(l,l')=\langle v_{l},v_{l'},s_{a,b}(x)\rangle,
\]
where $v_{l}$ spans the tangent direction of $L$ in $T(\P 5)_{x}$.

Notice that if we consider points $(a,b)\in\mathbb{C}^{2}\setminus\{(0,0)\}$
modulo a multiplicative constant, then for each $[a:b]\in\P 1$ we
get a different space $\langle v_{l},v_{l'},s_{a,b}(x)\rangle$. Indeed,
the canonical map
\[
T(L)_{x}\rightarrow N_{L/Y,x}
\]
is zero, so $\{v_{l},v_{l'},w_{0},w_{1}\}$ is linearly independent
in $T(\P 5)_{x}$. Note that we lifted the generators of $N_{L/Y,x}$
using \eqref{eq:Diaglift}. Obviously, the vector space generated
by $\{v_{l},v_{l'},s_{a,b}(x)\}$ depends only on $[a:b]\in\P 1$.
In particular the image of $(l,l')$ depends on the curve $C$. This
implies that $\rho_{|\FF}$ cannot be extended to a general point
of $I$ and we are done.
\end{proof}

\section{General facts about the LLSS variety}
\begin{defn}
A rational normal curve of degree 3, or twisted cubic for short, is
a smooth curve $C\subset\P 3$ that is projectively equivalent to
the image of $\P 1$ under the Veronese embedding $\P 1\ra\P 3$ of
degree 3. 
\end{defn}

If $X\subseteq\bP^{N}$ is a projective variety, we denote by $\Hilb{3z+1}(X)$
the Hilbert scheme of curves contained in $X$ with Hilbert polynomial
equal to $3z+1$.

Piene and Schlessinger \cite{MR796901} showed that $\Hilb{3z+1}(\P 3)=H_{0}\cup H_{1}$,
where $H_{0}$ is a 12-dimensional smooth irreducible component such
that the general point is a rational normal curve, and $H_{1}$ is
a 15-dimensional smooth irreducible component such that the general
point is a curve $C$ such that $C_{red}$ is a plane cubic. A generalized
twisted cubic is a subscheme in $\P 3$ represented by a point in
$H_{0}$.

We define $\Hilb{gtc}(\P 3)$ as the subscheme $\Hilb{gtc}(\P 3):=H_{0}\subset\Hilb{3z+1}(\P 3)$.
It is known that $\Hilb{3z+1}(\P 5)$ contains a smooth irreducible
component $\Hilb{gtc}(\P 5)$ that parameterizes generalised twisted
cubics.

The general cubic fourfold $Y$ in $\P 5$ does not contain a plane,
see \cite[Section 1, Lemme 1]{MR860684}.
\begin{defn}
We set $F_{3}(Y):=\Hilb{gtc}(\P 5)\cap\Hilb{3z+1}(Y)$.
\end{defn}

By \cite[Theorem 4.7]{1305.0178} $F_{3}(Y)$ is a smooth variety
of dimension 10.

An element $[\Gamma]\in F_{3}(Y)$ is the class of a one dimensional
subscheme $\Gamma$ of $Y$ with Hilbert polynomial equal to $3z+1$.
The linear span of $\Gamma$ is a 3-dimensional space $\left\langle \Gamma\right\rangle \cong\P 3$,
and $[\Gamma]\in\Hilb{gtc}(\left\langle \Gamma\right\rangle )$. So
the span induces a morphism 
\[
F_{3}(Y)\ra G(4,6)
\]
and, as stated by \cite[p.113 and Theorem 4.8]{1305.0178}, there
is a commutative diagram

\begin{equation}
\xymatrix{F_{3}(Y)\ar[d]_{\phi}\ar[r] & G(4,6)\\
Z'\ar[ru]_{g}
}
\label{eq:ResolutionSteinSpan}
\end{equation}
where $Z'$ is a smooth irreducible projective variety. The diagram
(\ref{eq:ResolutionSteinSpan}) has the following remarkable properties: 
\begin{itemize}
\item The morphism $\phi$ is a $\P 2$-fibration. 
\item The morphism $g$ is finite on the open subset $g^{-1}(W_{ADE})=:V_{ADE}\subseteq Z'$
where 
\[
W_{ADE}:=\{P\in G(4,6)|P\cap Y\,\mathrm{has}\,ADE\,\mathrm{singularities\,or\,is\,smooth}\}.
\]
\item The degree of $g$ on $V_{ADE}$ is 72 by \cite[Theorem 2.1 and Table 1]{1305.0178}. 
\end{itemize}
Moreover, by \cite[Theorem 4.11 and Proposition 4.5]{1305.0178} there
exists a divisorial contraction 
\[
\xymatrix{Z'\ar[d]_{\sigma} & \bP(T_{Y})\ar[d]\ar@{_{(}->}[l]\\
Z & Y\ar@{_{(}->}[l]
}
\]
making $Z'$ the blow up of a variety $Z$ over a subvariety canonically
isomorphic to $Y$. By \cite[Theorem 4.19]{1305.0178}, $Z$ is an
hyperkähler variety.

Obviously, both $Z'$ and $Z$ depend on $Y$, so they should be denoted
by $Z'(Y)$ and $Z(Y)$. We choose to keep the same notation as \cite{1305.0178}.
So, when no confusion is possible, we will simply write $Z'$ and
$Z$.

\section{The Voisin map}

In \cite[Proposition 4.8]{MR3524175} Voisin defined a rational map
$\psi:\FF\rra Z$ using the following nice geometric argument. Let
$(l,l')\in\FF$ be a general point, that is, $l$ and $l'$ are the
classes of two disjoint lines $L$ and $L'$ such that the following
surface 
\[
S_{l,l'}:=\left\langle L,L'\right\rangle \cap Y
\]
is smooth. The point $(l,l')$ defines a linear system in $S_{l,l'}$
given by the divisor 
\[
D_{l,l'}=L-L'-K_{S_{l,l'}}.
\]
Since $\O_{S_{l,l'}}(1)=\O_{S_{l,l'}}(-K_{S_{l,l'}})$, in $|D_{l,l'}|$
there is a curve of the form $L\cup C'_{x}$, where $x$ is any point
of $L$ and $C'_{x}$ is the unique conic such that $\left\langle x,L'\right\rangle \cap Y=L'\cup C'_{x}$.
Then any member of this linear system is a generalized twisted cubic
contained in $Y$. Voisin defines the map $\psi$ by setting $\psi(l,l')$
to be the class in $Z$ of any member of $|D_{l,l'}|$. The degree
of the map is obtained as follows. It can be seen that $D_{l,l'}$
defines a morphism $\varphi_{D_{l,l'}}:S_{l,l'}\ra\P 2$ that contracts
exactly 6 lines. The members of $|D_{l,l'}|$ are pull-backs of lines
in $\P 2$. The line $L$ is the inverse image of a blown up point,
thus it is a component of the pull-back of any line through that point.
We can see, by intersection theory in $S_{l,l'}$, that $L'$ is the
strict transform of a conic through the other five points. Then we
have 6 lines that are components of some rational normal curve in
$|D_{l,l'}|$, so we have 6 possible choices of pairs of lines $R,R'\subseteq S_{l,l'}$
such that $|D_{l,l'}|=|D_{r,r'}|$.

We will describe Voisin's construction in full detail. Note that there
exists a rational map 
\[
\psi':\FF\rra Z',
\]
which differs from $\psi$ by a birational map, i.e., $\psi=\sigma\circ\psi'$.
In particular, by Voisin's construction of $\psi$ we already know
that $\psi'$ is dominant and has degree 6. In Proposition \ref{prop:psi'}
we will check that $\Ind{\psi'}=I$ and give a different argument
for the degree of $\psi'$.

First of all, we notice that if $L$ is a line and $C$ is a conic
in a projective space such that $L$ is not contained in the plane
defined by $C$ and if $L\cap C=\{x\}$, then $L\cup C$ is a limit
of rational normal curves \cite[Section 1.b p. 39]{MR685427}. If
both $L$ and $C$ are contained in $Y$, we have $[L\cup C]\in F_{3}(Y)$.
As already pointed out in \cite[Proposition 4.8]{MR3524175}, there
is a rational map 
\[
\begin{array}{c}
\psi_{1}:P\times F\dashrightarrow F_{3}(Y)\end{array}
\]
defined as follows. Let $(l,l')$ be not in $I$, let $x\in L$ be
a point and let $C'_{x}$ be the unique conic such that 
\[
\left\langle x,L'\right\rangle \cap Y=L'\cup C'_{x}.
\]
Then 
\[
\psi_{1}(l,x,l'):=[L\cup C'_{x}].
\]
Consider 
\[
U_{ADE}:=\left\{ (l,l')\in\FF\sd(l,l')\notin I,S_{l,l'}:=\left\langle L,L'\right\rangle \cap Y\,\mathrm{has}\,ADE\,\mathrm{singularities\,or\,is\,smooth}\right\} .
\]
Pick $(l,l')\in U_{ADE}$. Since the linear span of $L\cup C'_{x}$
is $\left\langle L,L'\right\rangle $ for each $x\in L$, then the
image of the curve 
\[
\Gamma_{l,l'}=\{[L\cup C_{x}]|x\in L\cong\P 1\}
\]
under the span map 
\[
F_{3}(Y)\ra G(4,6)
\]
is the point $\left\langle L,L'\right\rangle $. In other words the
curve $\Gamma_{l,l'}$ is contracted by $g\circ\phi$. By construction,
$g\circ\phi(\Gamma_{l,l'})\subset W_{ADE}$ so that 
\[
\phi(\Gamma_{l,l'})\subset V_{ADE}.
\]
The curve $\Gamma_{l,l'}$ must be contracted by $\phi$, since $\phi(\Gamma_{l,l'})$
is in the set where $g$ is finite. Let 
\[
p_{1}:P\times F\ra\FF
\]
be the canonical $\P 1$-bundle. Then the restricted map 
\[
\phi\circ\psi_{1}:p_{1}^{-1}(U_{ADE})\ra Z'
\]
contracts all the fibres of 
\[
p_{1}:p_{1}^{-1}(U_{ADE})\ra U_{ADE}.
\]
Since $p_{1}$ is a linear $\bP^{1}$-bundle, we can consider an open
set $U\subseteq U_{ADE}$ trivializing $p_{1}$ \cite[Section 3.2]{MR1318687}.
Then the diagram 
\begin{equation}
\xymatrix{p_{1}^{-1}(U)\cong\P 1\times U\ar[d]_{p_{1,U}}\ar[r]^{\,\quad\,\quad\,\quad\phi\circ\psi_{1}} & Z'\\
U
}
\label{eq:RigidityLemma}
\end{equation}
satisfies the hypothesis of the Rigidity Lemma \cite[Proposition 16.54]{MR2675155}.
Indeed: $U$ is reduced, $Z'$ is separated and $\P 1$ is reduced,
connected and proper. It follows that there exists a unique morphism
$U\ra Z'$ making (\ref{eq:RigidityLemma}) commutative. We can cover
$U_{ADE}$ by trivializing open subsets, repeat that argument and
get a map 
\[
U_{ADE}\ra Z'
\]
which point-wise is 
\[
(l,l')\mapsto\phi([L\cup C_{x}]).
\]
Because we know that this map does not depend on the choice of $x\in L$.
We have therefore defined a rational map 
\[
\begin{array}{cccc}
\psi': & \FF & \rra & Z'\\
 & (l,l') & \mapsto & \phi([L\cup C'_{x}]),
\end{array}
\]
for all $(l,l')$ in the open subset $U_{ADE}$ of $\FF$. As we said
before, the map $\psi'$ is dominant of degree 6 because these properties
hold for $\psi$, as Voisin proved. In the following proposition,
we will see how the degree of $\psi'$ can be obtained also from the
map $g:Z'\ra G(4,6)$.
\begin{prop}
\label{prop:psi'}The rational map $\psi':\FF\rra Z'$ defined above
is dominant, has degree 6 and $\Ind{\psi'}=I$.
\end{prop}

\begin{proof}
The composition of $\psi'$ with $g:Z'\ra G(4,6)$ gives rise to a
commutative diagram

\begin{equation}
\xymatrix{F\times F\ar@{-->}[d]_{\psi'}\ar@{-->}[r]^{\rho_{|\FF}} & G(4,6)\\
Z'\ar[ru]_{g}
}
\label{eq:FFZ'G}
\end{equation}
where $\rho$ is the span map. The inclusion $I\subseteq\Ind{\psi'}$
is an application of Lemma \ref{lem:CommDiagRatMaps} and Proposition
\ref{prop:IndspanI} to the Diagram (\ref{eq:FFZ'G}). To prove the
other inclusion, let 
\[
p_{1}:P\times F\ra\FF
\]
be the canonical $\P 1$-bundle and consider the diagram

\begin{equation}
\xymatrix{P\times F\ar[d]_{p_{1}}\ar@{-->}[r]^{\psi_{1}} & F_{3}(Y)\ar[d]_{\phi}\\
F\times F\ar@{-->}[r]^{\psi'} & Z'.
}
\label{eq:Diag-1}
\end{equation}
All maps in this diagram are defined on $p_{1}^{-1}(U_{ADE})$ and
the diagram commutes there. Indeed, if $(l,x,l')\in p_{1}^{-1}(U_{ADE})$,
then we know that $\psi_{1}$ is defined in $(l,x,l')$ as $(l,l')\notin I$.
Then 
\[
\phi(\psi_{1}(l,x,l'))=\phi([L\cup C_{x}])=\phi(\Gamma_{l,l'})=\psi'(l,l')=\psi'(p_{1}(l,x,l')).
\]
Consider a local section $s:\FF\ra P\times F$ of $p_{1}$, and let
\[
\psi'':=\phi\circ\psi_{1}\circ s
\]
be the rational map defined on $\FF\less I$. By commutativity of
(\ref{eq:Diag-1}), $\psi''$ coincides with $\psi'$ on $U_{ADE}$.
This implies that $\psi'$ can be extended to every point of $\FF\less I$.
We have then proved that 
\[
\Ind{\psi'}\subseteq I.
\]
Let $M\in G(4,6)$, then 
\[
\rho^{-1}(M)=\left\{ (l,l')\in\FF\less I|\left\langle L,L'\right\rangle =M\right\} .
\]
In particular, the pairs $(l,l')\in\rho^{-1}(M)$ represent pairs
of disjoint lines contained in the cubic surface $M\cap Y$. If $M$
is sufficiently general, then $M\cap Y$ is smooth and contains 27
lines, each of them meeting exactly 10 other lines \cite[pag 485]{MR507725}.
It can easily be seen that there are $27\cdot(27-11)=432$ pairs of
lines contained in $\rho^{-1}(M)$, and therefore $\rho$ is generically
finite of degree 432. All the manifolds appearing in the commutative
diagram \eqref{eq:FFZ'G} are 8-dimensional. Hence, as $\rho$ is
generically finite, it is also dominant. This implies that also $g$
and $\psi'$ are dominant and generically finite. Since $\deg\rho=\deg{\psi'}\deg g$
by commutativity of (\ref{eq:FFZ'G}), the degree of $\psi'$ is 
\[
\deg\psi'=\frac{432}{72}=6.
\]
\end{proof}
\begin{defn}
\label{def:Resolution}Let $\psi$ be the Voisin map. A resolution
of the indeterminacy of the map $\psi$ is a commutative diagram 
\begin{equation}
\xymatrix{\widetilde{F\times F}\ar[d]_{\pi}\ar[dr]^{\wt{\psi}}\\
F\times F\ar@{-->}[r]^{\psi} & Z
}
\label{eq:PsiPitPsi}
\end{equation}
where $\wt{\FF}$ is a non-singular variety and $\pi$ is a birational
morphism that is an isomorphism outside $\Ind{\psi}$. 
\end{defn}

We will denote simply by $\wt{\psi}:\widetilde{\FF}\ra Z$ a fixed
resolution of the indeterminacy of $\psi$. Moreover, we denote by
$E$ the support of the exceptional divisor of $\pi$. The existence
of such a resolution follows by \cite[I.Question (E) p.140]{MR0199184}.
The map $\pi$ may be obtained as a sequence of blow-ups along smooth
subvarieties. 
\begin{rem}
Notice that we have the following commutative diagram.

\begin{equation}
\xymatrix{P\times F\ar[dd]_{p_{1}}\ar@{-->}[rr]^{\psi_{1}} &  & F_{3}(Y)\ar[dd]^{\sigma\circ\phi}\ar[dl]_{\phi}\\
 & Z'\ar[dr]^{\sigma}\\
F\times F\ar@{-->}[rr]^{\psi}\ar@{-->}[ur]^{\psi'} &  & Z
}
\label{eq:Diag-1-1}
\end{equation}
Hence, the Voisin map defined in \cite[Proposition 4.8]{MR3524175}
is the composition $\sigma\circ\psi'$. 
\end{rem}

For the reader's convenience, we collect the following. 
\begin{lem}[{{\cite[Remark 4.10]{MR3524175}}}]
\label{lem:psiFinite}The map $\psi:\FF\rra Z$ is étale of degree
6 where it is defined. Furthermore, the image of the exceptional divisor
of the resolution $\pi$ of $\psi$ contains a divisor.
\end{lem}

\begin{proof}
The map $\psi$ is dominant of degree 6 because it is the composition
of a dominant degree 6 rational map and of a blow up. Let $R_{\td{\psi}}$
be the ramification divisor of $\td{\psi}$, that is the divisor supported
in the subset of points of $\wt{\FF}$ where the induced map $d\td{\psi}:T_{\wt{\FF}}\ra\td{\psi}^{*}T_{Z}$
is not an isomorphism. The scheme structure is given locally by the
vanishing of the Jacobian determinant $\det d\td{\psi}$, see \cite[Example 3.2.20]{MR1644323}.
Thus we have the formula 
\[
K_{\widetilde{F\times F}}=\pi^{*}K_{F\times F}+E'=\tilde{\psi}^{*}K_{Z}+R_{\td{\psi}}
\]
and since the first Chern class of $F$ and $Z$ is trivial, $E'=R_{\td{\psi}}$.
This implies that the ramification locus of $\td{\psi}$ is $E=\mathrm{Supp}E'$,
so that the Jacobian matrix is of maximal rank outside $E$, in other
words, $\psi$ is étale where it is defined. Let $D=\tilde{\psi}(E)$
be the image of the exceptional divisor. If we denote $G:=\tilde{\psi}^{-1}(D)\supseteq E$,
then $\tilde{\psi}_{|\widetilde{F\times F}\less G}:\widetilde{F\times F}\less G\ra Z\less D$
is a nontrivial topological cover of degree 6. Since $Z$ is simply
connected \cite[Theorem 4.19]{1305.0178}, if $D$ does not contain
a divisor then by \cite[Chap. X Th\' eor\` eme 2.3]{MR0301725} also
$Z\less D$ is simply connected, so it cannot have nontrivial topological
cover. Hence, an irreducible component of $D$ must be a divisor and
we are done. 
\end{proof}
We are now ready for the following. 
\begin{proof}[Proof of the Theorem \ref{thm:Main}]
From the commutative diagram

\begin{equation}
\xymatrix{F\times F\ar@{-->}[d]_{\psi'}\ar@{-->}[r]^{\psi} & Z\\
Z'\ar[ru]_{\sigma}
}
\label{eq:FFZ'G-1}
\end{equation}
we have the inclusion 
\begin{equation}
\Ind{\psi}\subseteq I,\label{eq:IndinI}
\end{equation}
by Lemma \ref{lem:CommDiagRatMaps} and Proposition \ref{prop:psi'}.
Since $\sigma:Z'\ra Z$ is a blow up along $Y\subseteq Z$, then on
$Z\less Y$ we can compose its inverse with $g:Z'\ra G(4,6)$ and
get a map 
\[
g_{|Z}:Z\less Y\ra G(4,6).
\]
Let $\tilde{\psi:}\widetilde{\FF}\ra Z$ be as in Definition \ref{def:Resolution},
and set 
\begin{equation}
W:=\pi(\tilde{\psi}^{-1}(Y)).\label{eq:W}
\end{equation}
We point out that outside the exceptional divisor $E$, the map $\tilde{\psi}$
is quasi-finite by Lemma \ref{lem:psiFinite}. It follows that there
is a commutative diagram of rational maps 
\begin{equation}
\xymatrix{F\times F\less W\ar@{-->}[d]_{\psi_{|\FF\less W}}\ar@{-->}[rr]^{\rho_{|\FF\less W}} &  & G(4,6)\\
Z\less Y.\ar[rru]_{g_{Z}}
}
\label{eq:FFZYG}
\end{equation}
We argue by contradiction. Suppose $\Ind{\psi}\subsetneq I$, and
set 
\[
T:=\pi(\pi^{-1}(I)\less(\wt{\psi}^{-1}(Y)\cap\pi^{-1}(I))).
\]
\begin{claim*}
The set $T$ is dense in $I$.
\end{claim*}
\begin{proof}[Proof of the Claim]
By Lemma \ref{lem:IirrCod2} and by our assumption $\Ind{\psi}\subsetneq I$,
the set $I\less\Ind{\psi}$ is dense open in $I$. It follows that
also 
\[
\wt I=\pi^{-1}(I\less\Ind{\psi})
\]
is dense in $\pi^{-1}(I)$, since $\pi$ is an isomorphism outside
$\Ind{\psi}$. Now $\wt I\less(\wt{\psi}^{-1}(Y)\cap\wt I)$ is dense
in $\wt I$ since it is open and not empty. Indeed, if it were empty
then $\wt I\subseteq\wt{\psi}^{-1}(Y)$, so 
\[
\wt{\psi}(\wt I)\subseteq Y.
\]
That is impossible since $Y$ has dimension 4 and $\wt{\psi}(\wt I)$
has dimension 6 (because $\wt I$ is contained in the open set where
$\wt{\psi}$ is quasi-finite). Since the map $\pi_{|\wt I}:\wt I\ra I$
is dominant, the image of the dense open set $\wt I\less(\wt{\psi}^{-1}(Y)\cap\wt I)$
is dense. But 
\begin{eqnarray*}
\pi(\wt I\less(\wt{\psi}^{-1}(Y)\cap\wt I)) & = & \pi(\pi^{-1}(I\less\Ind{\psi})\less(\wt{\psi}^{-1}(Y)\cap\pi^{-1}(I\less\Ind{\psi})))\\
 & \subseteq & \pi(\pi^{-1}(I)\less(\wt{\psi}^{-1}(Y)\cap\pi^{-1}(I)))=T,
\end{eqnarray*}
so that $T$ is dense. 
\end{proof}
Thus there exists a point 
\begin{equation}
(l,l')\in T\cap(I\less\Ind{\psi}),\label{eq:point(l,l')}
\end{equation}
and since $\pi$ is an isomorphism outside $\Ind{\psi}$, there exists
a unique $u\in\widetilde{\FF}$ such that 
\begin{equation}
\pi(u)=(l,l').\label{eq:UnicityU}
\end{equation}
By definition of $T$, the point $u$ is not in $\wt{\psi}^{-1}(Y)$.
If we apply Lemma \ref{lem:CommDiagRatMaps} to Diagram (\ref{eq:FFZYG})
we get 
\[
\Ind{\rho_{|\FF\less W}}\subseteq\Ind{\psi_{|\FF\less W}}
\]
and thus 
\[
I\less(W\cap I)\subseteq\Ind{\psi}\less(W\cap\Ind{\psi}).
\]
In particular, we have a dense subset of $I$ contained in $\Ind{\psi}$,
so that $I\subseteq\Ind{\psi}$. Hence by (\ref{eq:IndinI}) we get
$I=\Ind{\psi}$. This contradicts the assumption $\Ind{\psi}\subsetneq I$
and we are done. 
\end{proof}
\bibliographystyle{amsalpha}
\bibliography{biblo}
 
\end{document}